\documentclass[10pt,letterpaper]{amsart}

\usepackage{amsthm,amsmath,amssymb,amsfonts}
\usepackage{hyperref}
  \hypersetup{colorlinks=true,linkcolor=blue,citecolor=red,urlcolor=blue}
\usepackage[utf8x]{inputenc}
\usepackage[all]{xy}
\usepackage{enumerate}
\usepackage{stmaryrd}

\newtheorem{theorem}{Theorem}

\newtheorem{lemma}[theorem]{Lemma}

\theoremstyle{definition}
  \newtheorem{definition}[theorem]{Definition}
  \newtheorem{remark}[theorem]{Remark}

\title{An example of a non--algebraizable singularity}
\author{Valente Ram\'{i}rez}
\address{Dept.~of Mathematics, Cornell University, 120 Malott Hall, Ithaca, NY 14850, USA.}
\email{valente@math.cornell.edu}
  \hypersetup{
  pdftitle={An example of a non-algebraizable singularity},
  pdfauthor={Valente Ramirez}
  }

\newcommand{\R}{\mathbb{R}}
\newcommand{\Q}{\mathbb{Q}}
\newcommand{\C}{\mathbb{C}}
\renewcommand{\P}{\mathbb{P}}

\newcommand{\F}{\mathcal{F}}
\newcommand{\GL}[2]{\operatorname{GL}({#1},{#2})}

\begin{document}

\begin{abstract}
 We say that the the germ of a singular holomorphic foliation on $(\C^2,0)$ is algebraizable whenever it is holomorphically conjugate to the singularity of a foliation defined globally on a projective algebraic surface. The object of this work is to construct a concrete example of a non--algebraizable singularity. Previously only existential results were known \cite{GenzmerTeyssier2010}.
\end{abstract}

\thanks{This work was supported by \emph{Laboratorio Internacional Solomon Lefschetz (LAISLA)} associated to CNRS (France) and CONACYT (Mexico) and the grant UNAM--DGAPA--PAPIIT IN 102413 (Mexico).}
\maketitle

\section{Introduction}

Let $\F$ be the germ of a holomorphic foliation on $(\C^2,0)$ with an isolated singularity. We are interested in understanding whether there exists or not an algebraic surface $S$ and a point $p$ on it such that $\F$ is holomorphically conjugate to the germ at $p$ of an algebraic foliation on $S$. Those germs for which such an algebraic foliation exists are called \emph{algebraizable}. The existence of non--algebraizable singularities remained unknown until Genzmer and Teyssier proved in \cite{GenzmerTeyssier2010} the existence of countably many classes of saddle--node singularities which are not algebraizable. Their proof however, does not provide us with any concrete examples of such singularities and, as far as the author knows, no other examples of non--algebraizable singularities are known. Following Casale \cite{Casale2013} we split the problem into two parts: First, give an example of a germ of a non--algebraizable singularity; second, identify algebraizable singularities. In this paper we address the first question and construct explicitly the germ of a degenerate singularity of order two on $(\C^2,0)$ which is not algebraizable.

The strategy we shall follow to construct our example is based on the following observation: Any algebraic singularity depends on finitely many complex parameters and so these parameters necessarily generate a field extension of $\Q$ with finite transcendence degree. In order to build a non--algebraizable singularity we seek to define a foliation whose defining coefficients carry \emph{enough transcendence} in some intrinsic way. In order to formalize the previous sentence we make use of the formal classification of non--dicritic degenerate singularities given by Ortiz--Bobadilla, Rosales--Gonz\'{a}lez and Voronin in \cite{OrtizRosalesVoronin2012}. More precisely, we construct a holomorphic 1--form which is in the formal normal form introduced in \cite{OrtizRosalesVoronin2012} in such a way that its \emph{radial part}, a formal invariant, is defined by a power series with coefficients generating a field of infinite transcendence degree. The non--algebraizability of such foliation follows from the observation that any polynomial 1--form generating a non--dicritic degenerate singularity has a formal normal form as in \cite{OrtizRosalesVoronin2012} whose coefficients generate a field of finite transcendence degree. Our main result is stated as follows.

\begin{theorem}\label{thm:main}
 Let $\lambda_1,\lambda_2,\lambda_3$ be non--rational numbers satisfying $\lambda_1+\lambda_2+\lambda_3=1$ and let $f_j=a_jx+b_jy$, $j=1,2,3$, be different linear forms in $\C[x,y]$. Define the homogeneous quadratic 1--form 
 \[ \omega_0=f_1f_2f_3\sum_{j=1}^{3}\lambda_j\frac{df_j}{f_j}. \]
 Let $\mathcal{B}=\{b_0,b_1,\ldots\}$ be a subset of $\C$ such that the field extension $\Q(\mathcal{B})/\Q$ has infinite transcendence degree and such that the power series
 \[ b(x)=\sum_{k=0}^{\infty}b_k x^k \]
 has a positive radius of convergence. Then the germ of the holomorphic foliation on $(\C^2,0)$ defined by the 1--form
 \begin{equation}\label{eq:example}
 \omega=\omega_0+x^2b(x)(xdy-ydx) 
 \end{equation}
 is not algebraizable.
\end{theorem}

\begin{remark}\label{rmk:LindemannWeierstrass}
 In virtue of the Lindemann--Weierstrass theorem if $\{a_1,a_2,\ldots\}$ is a collection of algebraic numbers spanning an infinite--dimensional vector space over $\Q$ then the set $\{e^{a_1},e^{a_2},\ldots\}$ generates a field extension of $\Q$ of infinite transcendence degree. This gives us an immense amount of flexibility defining the set $\mathcal{B}$. In particular we can choose $\mathcal{B}\subset\R$ and we can make make $b(x)$ as rapidly convergent as desired. For example, defining
 \[ b_k=e^{-k\sqrt{k}},\qquad k=0,1,2,\ldots, \]
 gives rise to an entire function $b(x)=\sum_{k=0}^{\infty} b_k x^k$.
\end{remark}

It is worth remarking that, on the other hand, we do have a few criteria for deciding algebraizability. It is known since Poincar\'{e} and Dulac that non--degenerate planar singularities with spectrum on the so--called \emph{Poincar\'{e} domain} are analytically equivalent to foliations given by a polynomial 1--form on $\C^2$. In addition, Casale proved in \cite{Casale2013} that the class of dicritical foliations on $(\C^2,0)$ which are regular after a single blow--up and have a unique leaf tangent to the exceptional divisor are algebraic whenever they admit a meromorphic first integral. More recently Calsamiglia and Sad \cite{CalsamigliaSad2014} generalized this result to the class of all dicritic foliations which are regular after one blow--up process, thus removing the requirement of a single tangency with the exceptional divisor.

\subsection*{Acknowledgements}

This result was obtained during a visit to the \emph{Institut de Recherche Math\'{e}matique de Rennes (IRMAR)}. I wish to thank Frank Loray for suggesting this problem and for all the fruitful conversations that led to this paper. I'm particularly thankful to Laura Ortiz who made this visit possible.

\section{Formal classification of non--dicritic singularities}

We wish to construct a 1--form $\omega$ on $(\C^2,0)$ which cannot be conjugate to a 1--form defined by coefficients contained in a field of finite transcendence degree. Having $\omega$ defined by power series with coefficients generating a field of infinite transcendence degree is not enough since those coefficients will not be preserved under an arbitrary change of coordinates. We are thus led to seek for a foliation carrying enough transcendence not on its defining coefficients but on its \emph{formal invariants}. For the sake of briefness we state the theorem on formal classification (cf.~Theorem 1.1 and Corollary 1.4 in \cite{OrtizRosalesVoronin2012}) only for degenerate singularities of order two. For our own convenience we state their result in terms of differential 1--forms and not vector fields; the adaptation is straight forward.

\begin{theorem}[\cite{OrtizRosalesVoronin2012}]\label{thm:ORV}
 A generic 1--form $\eta$ on $(\C^2,0)$ having a degenerate singularity of order two is formally equivalent to a formal 1--form $\hat{\eta}$ of the form
 \begin{equation}\label{eq:FNF} 
 \hat{\eta}=\eta_0+x^2b(x)(xdy-ydx), 
 \end{equation}
 where $\eta_0$ is the quadratic homogeneous part of $\eta$ and $b(x)\in\C\llbracket x\rrbracket$. Such normal form is unique up to homotheties and multiplication by a scalar factor.
\end{theorem}

Let $R(x,y)=x\frac{\partial}{\partial x}+y\frac{\partial}{\partial y}$ be the radial vector field and let $\eta_0$ denote the quadratic homogeneous part of $\eta$ as above. The \emph{tangent cone} of $\eta$ is defined to be the polynomial $\eta_0(R)$ and, by definition, $\eta$ is non--dicritic if $\eta_0(R)\not\equiv0$.

\begin{definition}\label{def:genericity}
 In the above theorem and throughout this text we will say that a 1--form $\eta$ having a degenerate singularity of order two is \emph{generic} if it satisfies
  \begin{enumerate}[(i)]
   \item The tangent cone $T=\eta_0(R)$ is non--zero and has three simple linear factors $l_1,l_2,l_3$.
   \item The residues $\alpha_1,\alpha_2,\alpha_3$ of the meromorphic 1--form 
    \[ \frac{\eta_0}{T}=\sum_{j=1}^{3}\alpha_j\frac{dl_j}{l_j} \]
   are not rational numbers.
  \end{enumerate}
\end{definition}

\begin{remark}\label{rmk:FixTangentCone}
 A generic tangent cone as above decomposes into three linear factors which, by a linear change of coordinates, can be normalized to be $y$, $x$ and $x-y$. From now on we will assume that all foliations under consideration have $T(x,y)=xy(x-y)$ as tangent cone.
\end{remark}

\section{Proof of the Theorem}

We shall prove first that the 1--form defined in \hyperref[thm:main]{Theorem \ref{thm:main}} is not equivalent to a polynomial 1--form on $\C^2$. Once we do this it will follow easily that such a singularity cannot be conjugate to an algebraic singularity on a projective surface.

\begin{lemma}\label{lemma:reduction}
 Let $\mathbb{K}$ be a subfield of $\C$ and let $P,Q\in\mathbb{K}\llbracket x,y\rrbracket$ be formal power series. Assume $\eta=Pdx+Qdy$ defines a generic singularity of order two. The formal reduction taking $\eta$ to its formal normal form \textnormal{(\ref{eq:FNF})} is given by a formal map defined over the field $\mathbb{K}$.
\end{lemma}

\begin{proof}
 This proposition follows almost immediately from the proof of the formal classification theorem provided in \cite{OrtizRosalesVoronin2012} where a \emph{pre--normalized} foliation (i.e.~a foliation whose separatrix tangent to the line $x=0$ has been rectified) is reduced to its formal normal form.

 Let us first show that, given the 1--form $\eta=Pdx+Qdy$ as above, we can rectify the separatrix tangent to $x=0$ by a formal change of coordinates defined over $\mathbb{K}$. We proceed recursively assuming the separatrix tangent to $x=0$ has been rectified up to jets of order $k$ (i.e.~$\eta\wedge dx\vert_{x=0}=O(y^{k+1})$) and define a formal change of coordinates $\phi_k$ that will rectify the separatrix up to jets of order $k+1$. In fact, it is enough to define a polynomial change of coordinates of the form
 \begin{equation}\label{eq:rectify}
  \phi_k(x,y)=(x+c_k\,y^k,y).
 \end{equation}
 A short computation shows that if the separatrix tangent to $x=0$ is indeed rectified up to jets of order $k$ then the above polynomial change of coordinates rectifies the separatrix up to jets of order $k+1$ for a suitable coefficient $c_k\in\C$. The condition that $\phi_k^\ast\eta$ has a straight separatrix up to jets of order $k+1$ is given by the equation
 \[ (\phi_k^\ast\, \eta)\wedge dx\vert_{x=0}=O(y^{k+2}), \]
 which reduces to a linear equation on $c_k$ whose coefficients belong to the field $\mathbb{K}$. In particular $c_k\in\mathbb{K}$ and the map $\phi_k$ is defined over $\mathbb{K}$. In this way we can fully rectify the separatrix by a sequence of maps of the form (\ref{eq:rectify}). To be more precise, any finite jet of the sequence of polynomial maps
  \[ \Phi_N=\phi_N\circ\ldots\circ\phi_2, \quad N=2,3,\ldots, \]
 eventually stabilizes and thus we obtain a well defined formal map $\Phi=\lim_{N\to\infty}\Phi_N$ whose Taylor coefficients belong to the field $\mathbb{K}$.

 Because of the above paragraph we may assume without loss of generality that the 1--form $\eta$ in \hyperref[lemma:reduction]{Lemma \ref{lemma:reduction}} has a straight separatrix given by $x=0$. We can thus proceed with the formal reduction process given in \cite{OrtizRosalesVoronin2012}. In the aforementioned paper the form $\eta$ is brought to its formal normal form by a sequence of maps $H_k(x,y)$ followed by multiplication by functions $\mathcal{K}_k(x,y)$ of the form
 \[ H_k(x,y)=(x+\alpha_k(x,y)\, , \,y+\beta_k(x,y)), \qquad \mathcal{K}_k(x,y)=1-\delta_k(x,y), \]
 where $\alpha_k,\beta_k$ are homogeneous polynomials of degree $k$ and $\delta$ is a homogeneous polynomial of degree $k-1$. The coefficients of the polynomials $\alpha_k,\beta_k,\delta_k$ are obtained by solving a linear system of equations which are evidently defined over $\mathbb{K}$. This shows that the formal reduction process obtained in \cite{OrtizRosalesVoronin2012} is given by a formal map with coefficients in $\mathbb{K}$.
\end{proof}

\begin{proof}[Proof of Theorem 1]
 Let us prove prove \hyperref[thm:main]{Theorem \ref{thm:main}} by contradiction. Let $\eta=Pdx+Qdy$ be a polynomial 1--form on $\C^2$ and let $\mathbb{K}$ be the field generated by the coefficients of $P,Q\in\C[x,y]$ which necessarily has finite transcendence degree over $\Q$. Suppose $\eta$ is locally holomorphically equivalent to the 1--form $\omega$ given in (\ref{eq:example}) in a neighborhood of the origin. By \hyperref[lemma:reduction]{Lemma \ref{lemma:reduction}} there exists a formal normal form $\hat{\eta}$ of $\eta$ defined over the field $\mathbb{K}$. Since $\omega$ and $\hat{\eta}$ are formally equivalent and are in their formal normal form, \hyperref[thm:ORV]{Theorem \ref{thm:ORV}} implies that $\hat{\eta}$ and $\omega$ differ at most by a linear change of coordinates followed by multiplication by a scalar. Namely, there exits a linear map $A\in\GL{2}{\C}$ and a complex number $\lambda\in\C$ such that
 \begin{equation}\label{eq:comparison}
  A^{\ast}\hat{\eta}=\lambda\,\omega.
 \end{equation}
 This, however, is impossible since the left hand side of (\ref{eq:comparison}) is given by a power series with coefficients over a field of finite transcendence degree and the right hand side of (\ref{eq:comparison}) is given by a power series whose coefficients generate a field of infinite transcendence degree. We conclude that the 1--form $\omega$ cannot be conjugate to a polynomial 1--form on $\C^2$.

 Suppose now that there exists a foliation $\F$ on a smooth algebraic surface $S\subset \P^N$ such that the germ of the singularity defined by $\omega$ is holomorphically conjugate to the germ of $\F$ at a point $p\in S$. Since $p$ is a smooth point of $S$ we can find a general linear projection $f\colon\P^N\to\P^2$ such that the restriction $f\colon S\to\P^2$ is a branched covering map and $p$  a regular point of $f$. We use the germ of the biholomorphism $f\colon(S,p)\to(\P^2,f(p))$ to define the germ of a singularity
 \[ \widetilde{\F}=(f^{-1}\vert_{(\P^2,f(p))})^{\ast}\F, \]
 around $f(p)$ given by a (not necessarily polynomial) 1--form $\eta$. Without loss of generality, we may consider $\eta$ to be a 1--form on $(\C^2,0)$. Note that the map $f$, the foliation $\F$ and the surface $S$ are all defined by finitely many rational functions in $\C(\P^N)$. These rational functions are therefore defined over a field $\mathbb{K}$ of finite transcendence degree. Note that if $f\colon S\to\P^2$ is defined over the field $\mathbb{K}$ then the germ $f^{-1}\colon(\P^2,f(p))\to(S,p)$, which is guaranteed to exist by the inverse function theorem, is also defined over $\mathbb{K}$ since the inverse function theorem does not enlarge the field of definition $\mathbb{K}$. Neither does pulling back $\F$ by the map $f^{-1}$ will enlarge $\mathbb{K}$. This implies that $\eta$ is defined over a subfield $\mathbb{K}$ of $\C$ of finite transcendence degree over $\Q$ and is holomorphically equivalent to $\omega$, a contradiction.
\end{proof}

\bibliographystyle{alpha}
\bibliography{ref}

\end{document}